\documentclass{amsart}

\newcommand{\IN}{\mathbb N}
\newcommand{\w}{\omega}
\newcommand{\e}{\varepsilon}
\newcommand{\supp}{\mathrm{supp}}
\newcommand{\A}{\mathcal A}

\newtheorem{theorem}{Theorem}[section]
\newtheorem{corollary}[theorem]{Corollary}
\newtheorem{lemma}[theorem]{Lemma}
\newtheorem{proposition}[theorem]{Proposition}
\newtheorem{example}[theorem]{Example}
\newtheorem{problem}[theorem]{Problem}
\theoremstyle{definition}

\newtheorem{remark}[theorem]{Remark}

\title{Syndetic submeasures and partitions of $G$-spaces and groups}

\author{Taras Banakh, Igor Protasov, Sergiy Slobodianiuk}
\address{T.Banakh: Ivan Franko National University of Lviv, Ukraine and Jan Kochanowski University in Kielce, Poland.}
\address{I.Protasov and S.Slobodianiuk: Taras Shevchenko National University, Kyiv, Ukraine}
\email{t.o.banakh@gmail.com, i.v.protasov@gmail.com, slobodianiuk@yandex.ru}
\thanks{The first author has been partially financed by NCN grant  DEC-2011/01/B/ST1/01439.}

\begin{document}
\begin{abstract}
We prove that for every $k\in\IN$ each countable infinite group $G$ admits a partition $G=A\cup B$ into two sets which are {\em $k$-meager} in the sense that for every $k$-element subset $K\subset G$ the sets $KA$ and $KB$ are not thick. The proof is based on the fact that $G$ possesses a {\em syndetic submeasure}, i.e., a left-invariant submeasure $\mu:\mathcal P(G)\to[0,1]$ such that for each $\e>\frac1{|G|}$ and subset $A\subset G$ with $\mu(A)<1$ there is a set $B\subset G\setminus A$ such that $\mu(B)<\e$ and $FB=G$ for some finite subset $F\subset G$.
\end{abstract}

\maketitle

In this paper we continue the studies \cite{P1}--\cite{PZ} of combinatorial properties of partitions of $G$-spaces and groups.

By a {\em $G$-space} we understand a non-empty set $X$ endowed with a left action of a group $G$. The image of a point $x\in X$ under the action of an element $g\in G$ is denoted by $gx$. For two subsets $F\subset G$ and $A\subset X$ we put $FA=\{fa:f\in F,\;a\in A\}\subset X$.

\section{Prethick sets in partitions of $G$-spaces}

A subset $A$ of a $G$-space $X$ is called
\begin{itemize}
\item {\em large} if $FA=X$ for some finite subset $F\subset G$;
\item {\em thick} if for each finite subset $F\subset G$ there is a point $x\in X$ with $Fx\subset A$;
\item {\em prethick} if $KA$ is thick for some finite set $K\subset G$.
\end{itemize}

Now we insert number parameters in these definitions. Let $k,m\in\IN$. A subset $A$ of a $G$-space $X$ is called
\begin{itemize}
\item {\em $m$-large} if $FA=X$ for some subset $F\subset G$ of cardinality $|F|\le m$;
\item {\em $m$-thick} if for each finite subset $F\subset G$ of cardinality $|F|\le m$ there is a point $x\in X$ with $Fx\subset A$;
\item {\em $(k,m)$-prethick} if $KA$ is $m$-thick for some set $K\subset G$ of cardinality $|K|\le k$;
\item {\em $k$-prethick} if $KA$ is thick for some set $K\subset G$ of cardinality $|K|\le k$;
\item {\em $k$-meager} if $A$ is not $k$-prethick (i.e., $KA$ is not thick for any subset $K\subset G$ of cardinality $|K|\le k$).
\end{itemize}

In the dynamical terminology \cite[4.38]{HS}, large subsets are called syndetic and prethick subsets are called piecewise syndetic. We note also that these notions can be defined in much more general context of balleans \cite{PB}, \cite{PZ}.

The following proposition is well-known \cite[4.41]{HS}, \cite[1.3]{P4}, \cite[11.2]{PB}.

\begin{proposition}\label{p1} For any finite partition $X=A_1\cup\dots\cup A_n$ of a $G$-space $X$ one of the cells $A_i$ is prethick and hence $k$-prethick for some $k\in\IN$.
\end{proposition}

For finite groups the number $k$ in this proposition can be bounded from above by $n(\ln(\frac{|G|}n)+1)$. We consider each group $G$ as a $G$-space endowed the natural left action of $G$.

\begin{proposition}\label{p2}  Let $G$ be a finite group and $n,k\in\IN$ be numbers such that  $k\ge n\cdot\big(\ln(\frac{|G|}n)+1\big)$. For any $n$-partition $G=A_1\cup\dots\cup A_n$ of $G$ one of the cells $A_i$ is $k$-large and hence $k$-prethick.
\end{proposition}

\begin{proof} One of the cells $A_i$ of the partition has cardinality $|A_i|\ge\frac{|G|}{n}$. Then by \cite{Wein} or \cite[3.2]{BSR}, there is a subset $B\subset G$ of cardinality $|B|\le  \frac{|G|}{|A_i|}(\ln|A_i|+1)\le n(\ln(\frac{|G|}n) +1)\le k$ such that $G=BA_i$. It follows that the set $A_i$ is $k$-large and hence $k$-prethick.
\end{proof}

For $G$-spaces we have the following quantitative version of Proposition~\ref{p1}.

\begin{proposition}\label{p3} Let $m,n\in\IN$. For any $n$-partition $X=A_1\cup\dots\cup A_n$ of a $G$-space $X$ one of the cells $A_i$ is $(m^{n-1},m)$-prethick in $X$.
\end{proposition}

\begin{proof} For $n=1$ the proposition is trivial. Assume that it has been proved for some $n$ and take any partition $X=A_0\cup\dots\cup A_n$ of $X$ into $(n+1)$ pieces. If the cell $A_0$ is $(1,m)$-prethick, then we are done. If not, then there is a set $F\subset G$ of cardinality $|F|\le m$ such that $Fx\not\subset A_0$ for all $x\in X$. This implies that $x\in F^{-1}(A_1\cup\dots\cup A_n)$ and then by the inductive assumption, there is an index $1\le i\le n$ such that the set $F^{-1}A_i$ is $(m^{n-1},m)$-prethick. The latter means that there is a subset $E\subset G$ of cardinality $|E|\le m^{n-1}$ such that $EF^{-1}A_i$ is $m$-thick. Since $|EF^{-1}|\le|E|\cdot|F|\le m^{n-1}m=m^n$, the set $A_i$ is $(m^n,m)$-prethick.
\end{proof}

Looking at Proposition~\ref{p3} it is natural to ask what happens for $m=\w$. Is there any hope to find for every $n\in\IN$ a finite number $k_n$ such that for each $n$-partition $X=A_1\cup\dots\cup A_n$ some cell  $A_i$ of the partition is $k_n$-prethick? In fact, $G$-spaces with this property do exist.

\begin{example}\label{e4} Let $X$ be an infinite set endowed with the natural action of the group $G=S_X$ of all bijections of $X$. Then each subset $A\subset X$ of cardinality $|A|=|X|$ is $2$-large, which implies that for each finite partition $X=A_1\cup \dots \cup A_n$ one of the cells $A_i$ has cardinality $|A_i|=|X|$ and hence is $2$-large and $2$-prethick.
\end{example}

The action of the normal subgroup $FS_X\subset S_X$ consisting of all bijections $f:X\to X$ with finite support $\supp(f)=\{x\in X:f(x)\ne x\}$ has a similar property.

\begin{example}\label{e4a} Let $X$ be an infinite set endowed with the natural action of the group $G=FS_X$ of all finitely supported bijections of $X$. Then each infinite subset $A\subset X$ is thick, which implies that for each finite partition $X=A_1\cup \dots \cup A_n$ one of the cells $A_i$ is infinite and hence is thick and $1$-prethick.
\end{example}

However the $G$-spaces described in Examples~\ref{e4} and \ref{e4a} are rather pathological. In the next section we shall show that each $G$-space admitting a syndetic submeasure for every $k\in\IN$ can be covered by two $k$-meager (and hence not $k$-prethick) subsets. In Section~\ref{s6} using syndetic submeasures we shall prove that each countable infinite group admits a partition into two $k$-meager subsets for every $k\in\IN$.

\section{Syndetic submeasures on $G$-spaces}

A function $\mu:\mathcal P(X)\to[0,1]$ defined on the family of all subsets of a $G$-space $X$ is called \begin{itemize}
\item {\em $G$-invariant} if $\mu(gA)=\mu(A)$ for each $g\in G$ and a subset $A\subset X$;
\item {\em monotone} if $\mu(A)\le\mu(B)$ for any subsets $A\subset B\subset G$;
\item {\em subadditive} if $\mu(A\cup B)\le \mu(A)+\mu(B)$ for any sets $A,B\subset X$;
\item {\em additive} if $\mu(A\cup B)=\mu(A)+\mu(B)$ for any disjoint sets $A,B\subset X$;
\item a {\em submeasure} if $\mu$ is monotone, subadditive, and $\mu(\emptyset)=0$, $\mu(X)=1$;
\item a {\em measure} if $\mu$ is an additive submeasure;
\item a {\em syndetic submeasure} if $\mu$ is a $G$-invariant submeasure such that for each subset $A\subset X$ with $\mu(A)<1$ and each $\e>\frac1{|X|}$ there is a large subset $L\subset X\setminus A$ of submeasure $\mu(L)<\e$.
\end{itemize}
In this definition we assume that $\frac1{|X|}=0$ if the $G$-space $X$ is infinite.

\begin{proposition}\label{finite} A finite $G$-space $X$ possesses a syndetic submeasure if and only if $X$ is transitive.
\end{proposition}

\begin{proof} If $X$ is transitive, then the counting measure $\mu:\mathcal P(X)\to[0,1]$, $\mu:A\mapsto |A|/|X|$, is syndetic.

Now assume conversely that a finite $G$-space $X$ admits a syndetic submeasure $\mu:\mathcal P(X)\to[0,1]$. If $X$ is a singleton, then $X$ is transitive. So, we assume that $X$ contains more than one point. Since the empty set $A=\emptyset$ has submeasure $\mu(A)=0<1$, for the number $\e=\frac1{|X|-1}>\frac1{|X|}$ there is a large subset $L\subset X\setminus A=X$ of submeasure $\mu(L)<\e$. It follows that $L$, being large in $X$, has non-empty intersection with each orbit $Gx$, $x\in X$. Replacing $L$ by a smaller subset we can assume that $L$ meets each orbit in exactly one point. For every point $x\in L$ we can find a finite subset $F_x\subset G$ of cardinality $|Gx|-1$ such that $F_xx=Gx\setminus \{x\}$. Then the set $F=\{1_G\}\cup\bigcup_{x\in L}F_x$ has cardinality $|F|=1+\sum_{x\in L}(|G_x|-1)=1-|L|+\sum_{x\in L}|G_x|=1-|L|+|X|$ and $FL=X$. By the subadditivity and the $G$-invariance of the submeasure $\mu$, we get
$$1=\mu(FL)\le|F|\cdot\mu(L)<|F|\cdot\e=\frac{|F|}{|X|-1}=\frac{1-|L|+X}{|X|-1},$$which implies $|L|=1$. This means that $X$ has exactly one orbit and hence is transitive.
\end{proof}

For $G$-spaces admitting a syndetic submeasure we have the following result completing Propositions~\ref{p1}--\ref{p3}.

\begin{theorem}\label{t5} Let $G$ be a countable group and $X$ be an infinite $G$-space possessing a syndetic submeasure $\mu:\mathcal P(X)\to[0,1]$. Then for every $k\in\IN$ there is a partition $X=A\cup B$ of $X$ into two $k$-meager subsets.
\end{theorem}

\begin{proof} Fix any $k\in\IN$ and choose an enumeration $(K_n)_{n=1}^\infty$ of all $k$-element subsets of $G$.

Using the definition of a syndetic submeasure, we can inductively construct two sequences $(A_n)_{n=1}^\infty$ and $(B_n)_{n=1}^\infty$ of large subsets of $X$ satisfying the following conditions for every $n\in\IN$:
\begin{enumerate}
\item $A_n\subset X\setminus \bigcup_{i<n}K_n^{-1}K_iB_i$;
\item $\mu(A_n)< \frac1{k^22^n}$;
\item $B_n\subset X\setminus \bigcup_{i\le n}K_n^{-1}K_iA_i$;
\item $\mu(B_n)< \frac1{k^22^n}$.
\end{enumerate}
At each step the choice of the set $A_n$ is possible as
$$\mu(\bigcup_{i<n}K_n^{-1}K_iB_i)\le\sum_{i<n}\sum_{x\in K_n^{-1}K_i}\mu(xB_i)=\sum_{i<n}|K_n^{-1}K_i|\cdot\mu(B_i)\le\sum_{i<n}k^2\frac1{k^22^i}<1$$by the subadditivity of $\mu$. By the same reason, the set $B_n$ can be chosen.
\smallskip

After completing the inductive construction, we get the disjoint sets $A=\bigcup_{n=1}^\infty K_nA_n$ and $B=\bigcup_{n=1}^\infty K_nB_n$.

It remains to check that the sets $A$ and $X\setminus A$ are $k$-meager.
Given any $k$-element subset $K\subset G$ we need to prove that the sets $KA$ and $K(X\setminus A)$ are not thick. Find $n\in\IN$ such that $K_n=K^{-1}$.

Since the set $K_nB_n$ is disjoint with $A$, the large set $B_n$ is disjoint with $K_n^{-1}A=KA$,
which implies that $X\setminus KA$ is large and $KA$ is not thick.

Next, we show that the set $K(X\setminus A)=K_n^{-1}(X\setminus A)$ is not thick. We claim that  $A_n\subset X\setminus K_n^{-1}(X\setminus A)$. Assuming the converse, we can find a point $a\in A_n\cap K_n^{-1}(X\setminus A)$. Then $K_na$ intersects $X\setminus A$, which is not possible as $K_na\subset K_nA_n\subset A$. So, the set $X\setminus K(X\setminus A)\supset A_n$ is large, which implies that $K(X\setminus A)$ is not thick.
\end{proof}

\section{Toposyndetic submeasures on $G$-spaces}

In light of Theorem~\ref{t5} it is important to detect $G$-spaces possessing a syndetic submeasure.
We shall find such spaces among $G$-spaces possessing a toposyndetic submeasure. To define such submeasures, we need to recall some information from Measure Theory.

Let $\mu:\mathcal P(X)\to[0,1]$ be a submeasure on a set $X$. A subset $A\subset X$ is called {\em $\mu$-measurable} if $\mu(B)=\mu(B\cap A)+\mu(B\setminus A)$ for each subset $B\subset X$. By (the proof of) \cite[2.1.3]{Fed}, the family $\A_\mu$ of all $\mu$-measurable subsets of $X$ is an algebra (called the {\em measure algebra of $\mu$}) and the restriction $\mu|\A_\mu$ is additive in the sense that $\mu(A\cup B)=\mu(A)+\mu(B)$ for any disjoint $\mu$-measurable sets $A,B\in\A_\mu$.

A $G$-invariant submeasure $\mu:\mathcal P(X)\to[0,1]$ on a $G$-space $X$ will be called {\em toposyndetic} if $\A_\mu\cap\tau$ is a base of some $G$-bounded $G$-invariant regular topology $\tau$ on $X$. The {\em $G$-boundedness} of the topology $\tau$ means that each non-empty open set $U\in\tau$ is large in $X$. The $G$-boundedness of $\tau$ implies the density of all orbits $Gx$, $x\in X$, in the topology $\tau$.

\begin{theorem}\label{t3.1} If a $G$-space $X$ admits a toposyndetic submeasure, then each non-empty $G$-invariant subspace $Y\subset X$ possesses a syndetic submeasure.
\end{theorem}

\begin{proof} Let $\mu:\mathcal P(X)\to [0,1]$ be a toposyndetic submeasure on $X$ and $\tau$ be a  $G$-bounded $G$-invariant Tychonoff topology on $X$ such that $\A_\mu\cap\tau$ is a base of the topology $\tau$.

Fix any non-empty $G$-invariant subspace $Y\subset X$. The $G$-boundedness of the topology $\tau$ implies that $Y$ is dense in the topological space $(X,\tau)$. If the regular topological space $(X,\tau)$ has an isolated point $x$, then by the $G$-boundedness of the topology $\tau$ for the open set $U=\{x\}$ there is a finite set $F\subset G$ with $X=FU\subset Gx$, which means that $X$ is a finite transitive space. By the density of $Y$ in $X$, $Y=X$ and by Proposition~\ref{finite}, $Y$ possesses a syndetic submeasure.

So, we assume that the topological space $(X,\tau)$ has no isolated points.
The $G$-invariant submeasure $\mu$ induces a $G$-invariant submeasure $\lambda:\mathcal P(Y)\to[0,1]$ defined by $\lambda(A)=\mu(\bar A)$ for every subset $A\subset Y$, where $\bar A$ is the closure of $A$ in the topological space $(X,\tau)$.  To see that the submeasure $\lambda$ is syndetic, fix any $\e<\frac1{|Y|}=0$ and any subset $A\subset Y$ with $\lambda(A)<1$. Then $\mu(\bar A)=\lambda(A)<1$, which implies that $X\setminus \bar A$ is an open non-empty subset of $X$. Since $\A_\mu\cap\tau$ is a base of the topology $\tau$, there is a non-empty $\mu$-measurable open set $U\subset X\setminus \bar A\subset X\setminus A$. Since the topological space $(X,\tau)$ has no isolated points, we can fix pairwise disjoint non-empty open sets $U_1,\dots,U_n\subset U$ for some integer number $n>1/\e$. Since $\A_\mu\cap \tau$ is a base of the topology $\tau$, we can additionally assume that these open sets $U_1,\dots,U_n$ are $\mu$-measurable, which implies that $\sum_{i=1}^n\mu(U_i)\le 1$ and hence $\mu(U_i)\le \frac1n<\e$ for some $i\le n$. By the regularity of the topological space $(X,\tau)$, the open set $U_i$ contains the closure $\bar V$ of some non-empty open set $V\subset X$. The $G$-boundedness of $X$ guarantees that $V$ is large in $X$ and hence $V\cap Y$ is large in $Y$. Also $\lambda(V\cap Y)=\mu(\overline{V\cap Y})\le\mu(\bar V)\subset \mu(U_i)<\e$. This means that the submeasure $\lambda$ on $Y$ is syndetic.
\end{proof}

Many examples of $G$-spaces having a toposyndetic submeasure occur among subspaces of minimal compact measure $G$-spaces. By a {\em compact \textup{(}measure\textup{)} $G$-space} we understand a $G$-space $X$ endowed with a compact Hausdorff $G$-invariant topology $\tau_X$ (and a $G$-invariant probability Borel $\sigma$-additive measure $\lambda_X:\mathcal B(X)\to[0,1]$ defined on the $\sigma$-algebra $\mathcal B(X)$ of Borel subsets of $X$). A compact $G$-space $X$ is called {\em minimal} if each orbit $Gx$, $x\in X$, is dense in $X$.

\begin{theorem} If $(X,\tau_X,\lambda_X)$ is a minimal compact measure $G$-space, then each non-empty  $G$-invariant subspace $Y$ of $X$ possesses a (topo)syndetic submeasure.
\end{theorem}

\begin{proof} By the minimality of $X$, the $G$-invariant subspace $Y$ is dense in $X$.
Let $\tau=\{U\cap Y:U\in\tau_X\}$ be the induced topology on $Y$. The $G$-invariant measure $\lambda_X:\mathcal B(X)\to[0,1]$ induces a $G$-invariant submeasure $\mu:\mathcal P(Y)\to[0,1]$ defined by the formula $\mu(A)=\lambda_X(\bar A)$ for $A\subset Y$, where $\bar A$ denotes the closure of $A$ in the compact space $(X,\tau_X)$. To prove that the submeasure $\mu$ is toposyndetic, it remains to prove that the topology $\tau$ is $G$-bounded and $\A_\tau\cap\tau$ is a base of the topology $\tau$.

Consider the algebra $\A_X=\{A\subset X:\lambda_X(\partial A)=0\}$ consisting of subsets $A\subset X$ whose boundary $\partial A$ in $X$ have measure $\lambda_X(\partial A)=0$, and let $\A_Y=\{A\cap Y:A\in\A_X\}$. It can be shown that each set $A\subset\A_Y$ is $\mu$-measurable and $\A_Y\cap\tau\subset\A_\mu\cap\tau$ is a base of the topology $\tau$.
The $G$-boundedness of the topology $\tau$ on $Y$ is proved in the following lemma. Therefore, $\mu$ is a toposyndetic submeasure on $X$. By the proof of Theorem~\ref{t3.1}, the submeasure $\mu$ is syndetic.
\end{proof}

\begin{lemma}\label{l3.3} For each minimal compact $G$-space $X$, the induced topology on each $G$-invariant subspace $Y\subset X$ is $G$-bounded.
\end{lemma}

\begin{proof}  To show that the induced topology on $Y$ is $G$-bounded, fix any non-empty open subset $U\subset Y$. Find an open set $\widetilde U\subset X$ such that $\widetilde U\cap Y=U$.  By the regularity of the compact Hausdorff space $X$, there is a non-empty open subset $V\subset X$ with $\bar V\subset \tilde U$.

By a classical Birkhoff theorem in Topological Dynamics (see e.g. Theorem 19.26 \cite{HS}), the minimal compact $G$-space $X$ contains a uniformly recurrent point $y\in X$. The uniform recurrence of $y$ means that for each open neighborhood $O_y\subset X$ of $y$ the set $\{g\in G:gy\in O_y\}$ is large in $G$. By the density of the orbit $Gy$ there is $s\in G$ with $sy\in V$. Then $s^{-1}V$ is a neighborhood of $y$ and by the uniform recurrence of $y$, the set $L=\{g\in G:gy\in s^{-1}V\}$ is large in $G$. Consequently, we can find a finite subset $F\subset G$ such that $G=FL$. Then $Gy=FLy\subset Fs^{-1}V$, which implies that the open set $Fs^{-1}V$ is dense in $X$. Consequently, $X=Fs^{-1}\bar V\subset Fs^{-1}\tilde U$ and $Y=Fs^{-1}(Y\cap\tilde U)=Fs^{-1}U$, witnessing that the topology of $Y$ is $G$-bounded.
\end{proof}

\section{Groups possessing a toposyndetic submeasure}

In this section we shall detect groups possessing a toposyndetic submeasure. Each group $G$ will be considered as a $G$-space endowed with the natural left action of the group $G$. A group $G$ is called {\em amenable} if it admits a $G$-invariant additive measure $\mu:\mathcal P(G)\to[0,1]$.

We shall say that a $G$-space $X$ has a {\em free orbit} if for some $x\in X$ the map $\alpha_x:G\to X$, $\alpha_x:g\mapsto gx$, is injective.

\begin{theorem}\label{t4.1} A group $G$ admits a toposyndetic submeasure if one of the following conditions holds:
\begin{enumerate}
\item there is a minimal compact measure $G$-space $X$ with a free orbit;
\item $G$ is a subgroup of a compact topological group;
\item $G$ is countable;
\item $G$ is amenable.
\end{enumerate}
\end{theorem}

\begin{proof} 1. Assume that $(X,\tau_X,\lambda_X)$ is a minimal compact measure $G$-space with a free orbit. In this case there is a point $x\in X$ for which the map $\alpha_x:G\to Gx\subset X$, $\alpha_x:g\mapsto gx$, is injective. This map allows us to define a Tychonoff $G$-invariant topology $$\tau=\{\alpha^{-1}_x(U):U\in\tau_X\}$$ on the group $G$. By Lemma~\ref{l3.3}, the topology $\tau$ is $G$-bounded.

Since the orbit $Gx$ is dense in $X$ (which follows from the minimality of $X$), the formula
$$\mu(A)=\lambda(\overline{Ax})\mbox{ \ for \ }A\subset G$$
determines a $G$-invariant submeasure on $G$. Observe that $\mathcal B=\{U\in\tau_X:\lambda(\bar U)=\lambda(U)\}$ is a base of the topology $\tau_X$ on $X$ and $\A=\{\alpha_x^{-1}(U):U\in\mathcal B\}$ is a base of the topology $\tau$ on $G$. It can be verified that each set $A\in\A$ is $\mu$-measurable, which implies that $\A_\mu\cap\tau\supset\A$ is a base of the topology $\tau$. This means that the submeasure $\mu$ is toposyndetic.
\smallskip

2. The second statement follows immediately from the first statement and the well-known fact that each compact topological group carries an invariant probability Borel measure (namely, the Haar measure).
\smallskip

3. The third statement follows from the first one an a recent deep result of B.Weiss \cite{Weiss} stating that for each countable group $G$ there is a compact minimal measure $G$-space with a free orbit.
\smallskip

4. The fourth statement follows from the first statement and the well-known fact \cite[\S449]{Fremlin} stating for any amenable group $G$, each compact $G$-space $X$ possesses a $G$-invariant probability Borel measure.
\end{proof}

\begin{problem} Is the class of groups admitting a toposyndetic submeasure hereditary with respect to taking subgroups?
\end{problem}

\begin{problem} Has every group a toposyndetic submeasure?
\end{problem}

\begin{problem} Has the group $S_X$ of all bijections of an infinite set $X$ a toposyndetic submeasure?
\end{problem}

\section{Groups possessing a syndetic submeasure}

In this section we shall detect groups possessing a syndetic submeasure. By Theorem~\ref{t3.1} the class of such groups contains all groups possessing a toposyndetic submeasure, in particular, all countable groups.

\begin{theorem}\label{t5.1} A group $G$ possesses a syndetic submeasure if one of the following conditions is satisfied:
\begin{enumerate}
\item there is an infinite transitive $G$-space  possessing a syndetic submeasure;
\item there is an infinite minimal compact measure $G$-space;
\item $G$ admits a homomorphism onto an infinite group possessing a (topo)syndetic submeasure;
\item $G$ admits a homomorphism onto a countable infinite group;
\item $G$ contains an amenable infinite normal subgroup.
\end{enumerate}
\end{theorem}

\begin{proof} 1. Assume that $X$ is an infinite transitive $G$-space possessing a syndetic submeasure $\lambda:\mathcal P(X)\to[0,1]$. Fix any point $x\in X$ and consider the map $\alpha_x:G\to X$, $\alpha_x:g\mapsto gx$, which is surjective (by the transitivity of the $G$-space $X$). One can check that the syndetic submeasure $\lambda$ on $X$ induces a syndetic submeasure $\mu:\mathcal P(G)\to[0,1]$ defined by $\mu(A)=\lambda(\alpha_x(A))=\lambda_X(Ax)$ for $A\subset G$.
\vskip3pt

2. Let $(X,\tau_X,\mu_X)$ be an infinite minimal compact measure $G$-space. By the minimality, the orbit $Gx$ of any point $x\in X$ is dense in $(X,\tau_X)$. Then the formula $\mu(A)=\mu_X(\overline{Ax})$, $A\subset X$, determines a $G$-invariant submeasure $\mu:\mathcal P(G)\to[0,1]$ on the group $G$. We claim that the submeasure $\mu$ is syndetic. Given any $\e>\frac1{|G|}$ and a set $A\subset G$ with $\mu(A)<1$, we should find a large set $L\subset G\setminus A$ with $\mu(L)<\e$. Since $\mu_X(\overline{Ax})=\mu(A)<1$, the closed subset $\overline{Ax}$ is not equal to $X$. By the minimality, the infinite compact $G$-space $(X,\tau_X)$ has no isolated points, which allows us to find an open non-empty set $U\subset X\setminus\overline{Ax}$ such that $\mu_X(\overline{U})<\e$. By Lemma~\ref{l3.3}, the topology $\tau_X$ is $G$-bounded, which implies that the set $U\subset X$ is large in $X$ and hence $V=\alpha_x^{-1}(U)\subset X\setminus A$ is large in $G$ and has submeasure $\mu(V)\le\mu_X(\bar U)<\e$.
\vskip3pt

3. The third statement follows from the first statement and Theorem~\ref{t3.1}.
\vskip5pt

4. The fourth statement follows from the third statement and Theorem~\ref{t4.1}(3).
\vskip5pt

5. Suppose that the group $G$ contains a normal infinite amenable subgroup $H$. Denote by $P_\w(H)$ the set of finitely supported probability measures on $H$. Each measure $\mu\in P_\w(H)$ can be written as a convex combination $\mu=\sum_{i=1}^n \alpha_i\delta_{x_i}$ of Dirac measures concentrated at points $x_i$ of $H$. This allows us to identify $P_\w(H)$ with a convex subset
of the Banach space $\ell_1(H)$ endowed with the norm $\|f\|=\sum_{x\in H}|f(x)|$.

We claim that the function
$$\sigma_H:\mathcal P(G)\to[0,1],\;\;\sigma_H:A\mapsto \inf_{\mu\in P_\w(H)}\sup_{y\in G}\mu(Ay),$$is a syndetic left-invariant submeasure on $G$.

First we prove that $\sigma_H$ is left-invariant. Given any $x\in G$ and $A\subset G$ it suffices to check that $\sigma_H(xA)\le \sigma_H(A)+\e$ for every $\e>0$. The definition of $\sigma_H$ guarantees that $\sigma_H$ is right-invariant. Consequently, $\sigma_H(xA)=\sigma_H(xAx^{-1})$. By the definition of $\sigma_H(A)$, there is a finitely supported probability measure $\mu\in P_\w(H)$ such that $\sup_{y\in G}\mu(Ay)<\sigma_H(A)+\e$. Write $\mu$ as a convex combination $\mu=\sum_{i=1}^n\alpha_i\delta_{a_i}$ of Dirac measures concentrated at points $a_1,\dots,a_n\in H$. Since $H$ is a normal subgroup of $G$, the probability measure $\mu'=\sum_{i=1}^n\alpha_i\delta_{xa_ix^{-1}}$ belongs to $P_\w(H)$. Taking into account that for every $y\in G$
$$\mu'(xAx^{-1}y)=\mu'(xAx^{-1}yxx^{-1})=\mu(Ax^{-1}yx),$$ we conclude that
$$\sigma_H(xAx^{-1})\le \sup_{y\in G}\mu'(xAx^{-1}y)\le\sup_{y\in G}\mu(Ax^{-1}yx)<\sigma_H(A)+\e.$$
So, $\sigma_H$ is left-invariant.
\smallskip

Next, we prove that $\sigma_H$ is subadditive. Given two subsets $A,B\subset G$, it suffices to check that $\sigma_H(A\cup B)\le\sigma_H(A)+\sigma_H(B)+3\e$ for every $\e>0$. By the definition of the numbers $\sigma_H(A)$ and $\sigma_H(B)$, there are finitely supported probability measures $\mu_A,\mu_B\in P_\w(H)$ such that $\sup_{y\in G}\mu_A(Ay)<\sigma_H(A)+\e$ and $\sup_{y\in G}\mu_B(By)<\sigma_H(By)+\e$. By Emerson's characterization of amenability \cite[1.7]{Emerson}, for the probability measures $\mu_A$ and $\mu_B$ there are probability measures $\mu'_A,\mu'_B\in P_\w(H)$ such that $$\sup_{C\subset H}|\mu_A*\mu'_A(C)-\mu_B*\mu'_B(C)|\le\|\mu_A*\mu'_A-\mu_B*\mu_B'\|<\e.$$ Write the measures $\mu_A$, $\mu_B$, $\mu_A'$ and $\mu_B'$ as convex combinations of Dirac measures:
$$\mu_A=\sum_{i}\alpha_i\delta_{x_i},\;\mu_A'=\sum_{j}\alpha'_j\delta_{x'_j},\;
\mu_B=\sum_{i}\beta_i\delta_{y_i},\;\;\mu'_B=\sum_{j}\beta'_j\delta_{y'_i}.$$
Then $\mu_A*\mu'_A=\sum_{i,j}\alpha_i\alpha_j'\delta_{x_ix_j'}$ and $\mu_B*\mu'_B=\sum_{i,j}\beta_i\beta_j'\delta_{y_iy_j'}$. For every $y\in G$ we get
$$
\begin{aligned}
\mu_A*\mu'_A(Ay)&=\sum_{i,j}\alpha_i\alpha_j'\delta_{x_ix_j'}(Ay)=\sum_j\alpha_j'\sum_i\alpha_i\delta_{x_i}(Ay(x_j')^{-1})=\\
&=\sum_j\alpha_j'\,\mu_A(Ay(x_j)'^{-1})\le\sum_j\alpha_j'\sup_{z\in G}\mu_A(Az)=\sup_{z\in G}\mu_A(Az)<\sigma_H(A)+\e.
\end{aligned}
$$
By analogy we can prove that $\mu_B*\mu_B'(By)\le\sigma_H(B)+\e$. Now consider the measure $\nu=\mu_A*\mu_A'$ and observe that for every $y\in B$ we get
$$\nu(By)=\mu_A*\mu_A'(By)\le\mu_B*\mu_B'(By)+\|\mu_A*\mu_A'-\mu_B*\mu_B'\|< \sigma_H(B)+\e+\e.$$
Then $$\sigma_H(A\cup B)\le\sup_{y\in G}\nu((A\cup B)y)\le\sup_{y\in G}\nu(Ay)+\sup_{y\in G}\nu(By)<\sigma_H(A)+\e+\sigma_H(B)+2\e=\sigma_H(A)+\sigma_H(B)+3\e,$$
which proves the subadditivity of $\sigma_H$.
\smallskip

Finally we prove that the left-invariant submeasure $\sigma_H$ on $G$ is syndetic.
Fix a subset $A\subset G$ of submeasure $\sigma_H(A)<1$ and take an arbitrary $\e>0$. Since $\sigma_H(A)<1$, there is a finitely supported measure $\mu\in P_\w(H)$ such that $\sup_{y\in G}\mu(Ay)<1$. Write $\mu$ as the convex combination $\mu=\sum_{i=1}^n\alpha_i\delta_{x_i}$ of Dirac measures. We can assume that each coefficient $\alpha_i$ is positive. Then the finite set $F=\{x_1,\dots,x_n\}$ coincides with the support $\supp(\mu)$ of the measure $\mu$.

It follows that for every $y\in G$ we get $\mu(Ay)<1$ and hence $F=\supp(\mu)\not\subset Ay$. This ensures that the set $Fy^{-1}$ meets the complement $X\setminus A$ and hence $y^{-1}\in F^{-1}(G\setminus A)$. So,  $G=F^{-1}(G\setminus A)$ and the set $X\setminus A$ is large in $G$. Now take any finite subset $E\subset H$ of cardinality $|E|>1/\e$. Using Zorn's Lemma, choose a maximal subset $B\subset G\setminus A$ which is $E$-separated in the sense that $Ex\cap Ey=\emptyset$ for any distinct points $x,y\in B$. The maximality of the set $B$ guarantees that for each $x\in G\setminus A$ the set $Ex$ meets $EB$, which implies that $G\setminus A\subset E^{-1}EB$ and $G=F^{-1}(G\setminus A)=F^{-1}E^{-1}EB$. This means that the set $B$ is large in $G$. We claim that $|E^{-1}\cap By|\le 1$ for each $y\in G$. Assume conversely that $E^{-1}\cap By$ contains two distinct points $by$ and $b'y$ with $b,b'\in B$. Then $b'b^{-1}=b'y(by)^{-1}\in E^{-1}E$ and hence $Eb'\cap Eb\ne\emptyset$, which is not possible as $B$ is $E$-separated.
Now consider the uniformly distributed probability measure $\nu=\frac1{|E|}\sum_{x\in E^{-1}}\delta_x\in P_\w(H)$ and observe that $\sigma_H(B)\le\sup_{y\in G}\nu(By)\le \frac{|E^{-1}\cap By|}{|E|}\le\frac1{|E|}<\e$, which means that the submeasure $\sigma_H$ is syndetic.
\end{proof}

\begin{remark} For an infinite amenable group $G$ and the subgroup $H=G$ the syndetic submeasure  $\sigma_H$ (used in the proof of Theorem~\ref{t5.1}(5)) coincides with the right Solecki submeasure $\sigma^R$ introduced in \cite{Sol} and studied in \cite{Ban}.
\end{remark}

Theorem~\ref{t5.1}(5) implies:

\begin{corollary} The group $S_X$ of bijections of any set $X$ possesses a syndetic submeasure.
\end{corollary}

\begin{proof} If $X$ is finite, then the finite group $S_X$ has a syndetic submeasure according to proposition~\ref{finite}. So, we assume that the set $X$ is infinite. Observe that the subgroup  $FS_X$ of finitely supported permutations of $X$ is locally finite and hence amenable. By Theorem~\ref{t5.1}(5) the group $S_X$ admits a syndetic submeasure as it contains the infinite amenable normal subgroup $FS_X$.
\end{proof}

\begin{problem} Has every group a syndetic submeasure?
\end{problem}

\begin{problem} Has the quotient group $S_\w/FS_\w$ a syndetic submeasure?
\end{problem}

\section{Partitions of groups into $k$-meager pieces}\label{s6}

Now we return to the problem of partitioning groups into $k$-meager pieces, which was posed and partly resolved in \cite{PS}. Combining Theorems~\ref{t5} and \ref{t5.1}(5), we get:

\begin{theorem}\label{c7} Each countable infinite group $G$ for every $k\in\IN$ admits a partition into two $k$-meager subsets.
\end{theorem}

This theorem admits a self-generalization.

\begin{corollary}\label{c7a} If a group $G$ has a countable infinite quotient group, then for every $k\in\IN$ the group $G$ admits a partition into two $k$-meager subsets.
\end{corollary}

\begin{proof} Let $h:G\to H$ be a homomorphism of $G$ onto a countable infinite group $H$. By Theorem\ref{c7}, for every $k\in\IN$ the countable group $H$ admits a partition $H=A\cup B$ into two $k$-meager subsets. Then $G=h^{-1}(A)\cup h^{-1}(B)$ is a partition of the group $G$ into two $k$-meager subsets.
\end{proof}

Corollary~\ref{c7a} gives a partial answer to the following (still open) problem posed in and partially answered in \cite{PS}.

\begin{problem} Is it true that each infinite group $G$ for every $k\in\IN$ admits a partition into two  $k$-meager sets?
\end{problem}

\section{Acknowledgement}

The authors would like to express their sincere thanks to an anonymous referee who turned our attention to minimal measure-preserving actions of countable groups, which allowed us to prove the existence of toposyndetic submeasures on countable groups and construct partitions of such groups into $k$-meager sets.

\end{document}